\documentclass[oneside,10pt]{amsart}
\usepackage{amssymb, amscd}
\usepackage[all]{xy}
\usepackage{dsfont}

\usepackage{a4wide}
\usepackage[english]{babel}
\usepackage{amsfonts}
\usepackage{amsmath}
\usepackage{amsthm}
\title{Hilbert space compression under direct limits and certain group extensions}
\author{Dennis Dreesen}
\address{ K.U.Leuven campus Kortrijk\\
Etienne Sabbelaan 53\\
8500 Kortrijk}
\email{dennis.dreesen@kuleuven-kortrijk.be}
\address{Universit\'e de Neuch\^{a}tel\\
Institut de math\'ematiques\\
Rue Emile-Argand 11\\
2009 Neuch\^{a}tel}
\email{dennis.dreesen@unine.ch}
\thanks{The author is a research assistant for the Research Foundation - Flanders.}
\newtheorem{theorem}{Theorem}[section]
\newtheorem{definition}[theorem]{Definition}
\newtheorem{proposition}[theorem]{Proposition}
\newtheorem{lm}[theorem]{Lemma}
\newtheorem{corollary}[theorem]{Corollary}

\newtheorem{ex}[theorem]{Example}
\newtheorem{remark}[theorem]{Remark}

\newcommand{\Z}{\mathbb{Z}}
\newcommand{\N}{\mathbb{N}}
\newcommand{\R}{\mathbb{R}}

\newcommand{\h}{\mathcal{H}}

\newcommand{\Exp}{\mbox{Exp}}

\begin{document}
\begin{abstract}
We find bounds on the Hilbert space compression of the limit of a directed metric system of groups. We also give estimates on the Hilbert space compression of group extensions of a group $H$ by a a word-hyperbolic group or a group of polynomial growth.
\end{abstract}
\maketitle
\section{Introduction}
In \cite{Gromov}, Gromov introduced the notion of uniform embeddability of a finitely generated group into a Hilbert space and suggested that such a group would satisfy the Novikov Conjecture \cite{Gromov2}. Six years later, Yu came up with a formal proof of this claim \cite{Yu}. Moreover, together with Skandalis and Tu, he proved that such uniformly embeddable groups also satisfy the coarse Baum-Connes Conjecture \cite{Yu:2}.

\begin{definition} 
A metric space $(X,d)$ is {\em uniformly embeddable in a Hilbert space}, if there exist a Hilbert space $\h$, non-decreasing functions $\rho_-, \rho_+:\R^+ \rightarrow \R^+$ such that $\lim_{t\to \infty} \rho_-(t)=+\infty$, and a map $f:X\rightarrow \h$, such that
\[ \rho_-(d(x,y)) \leq d(f(x),f(y)) \leq \rho_+(d(x,y)) \ \forall x,y\in X.\]
The map $f$ is called a uniform embedding of $X$ in $\h$. It is called large-scale Lipschitz whenever $\rho_+$ can be taken of the form $\rho_+:t\mapsto Ct+D$ for some $C>0, D\geq 0$. It is Lipschitz if we can take $D=0$. \label{def:1}
\end{definition}

Two length functions $l_1$ and $l_2$ on a group (see Definition \ref{def:lengthfunction}) are coarsely equivalent if for every $R>0$ there exists $S>0$ such that the $l_1$-ball $B(1,R)$ with radius $R$ and center $1$ is contained in the $l_2$-ball $B(1,S)$; and conversely. Clearly in this case, $(X,l_1)$ is uniformly embeddable if and only if $(X,l_2)$ is uniformly embeddable. Lemma $2.1$ in \cite{Tu} shows that every discrete countable group $X$ admits a unique proper length function up to coarse equivalence, enabling us to define the concept of discrete countable uniformly embeddable group. This class of groups and its permanence properties have been very well studied, for example by Guentner and Dadarlat in \cite{guedad}. For our purposes and theirs, the following reformulation of their Proposition $2.1$, which holds for any metric space, is vital.

\begin{proposition}
Let $(X,d)$ be a metric space. Then $X$ is uniformly embeddable in a Hilbert space if and only if for every $n>0$ there exist $S_n >0$ and  a Hilbert space valued map $\xi_n:X\rightarrow \h, x\to \xi_n^x$ such that $\parallel \xi_n^x \parallel=1$ for all $x\in X$ and such that
\begin{enumerate}
\item $\parallel \xi_n^x - \xi_n^{x'} \parallel \leq \frac{1}{n}$ provided $d(x,x') \leq \sqrt{n}$,
\item $\parallel \xi_n^x - \xi_n^{x'} \parallel \geq 1$ provided $d(x,x')\geq S_n$.
\end{enumerate} \label{prop:unifembedintro}
\end{proposition}

The {\em speed at which $(S_n)_{n\in \N_0}$ tends to infinity} is an indication on {\em how uniformly embeddable} a metric space really is. For example, if $n\mapsto S_n$ is bounded by a polynomial map in $n$, it would make sense to say that the corresponding space is {\em more uniformly embeddable} than a space for which $n\mapsto S_n$ is only bounded by an exponential map in $n$. Another, more standard way of describing {\em how uniformly embeddable} a metric space $X$ really is, is by looking at the supremum of $\delta \geq 0$ such that there is a large-scale Lipschitz uniform embedding of $X$ and numbers $C',D'>0$ such that $\rho_-$ in Definition \ref{def:1} can be taken of the from
$r\mapsto \frac{1}{C'} r^\delta - D'$. This supremum is called the Hilbert space compression of $X$ \cite{guekam}. For $X$ a group, it must be noted that  Hilbert space compression is a quasi-isometric invariant, but no longer a coarse invariant. Therefore, it is important that we always specify the chosen length function.

Looking closely at the proof of Proposition \ref{prop:unifembedintro}, one finds a connection between the Hilbert space compression of a metric space and the growth of the sequence $(S_n)_{n\in \N_0}$. In this note, we try roughly to exploit this connection and then use techniques from \cite{guedad} to get concrete information about the behaviour of the Hilbert space compression of groups under taking direct limits and under taking certain group extensions. Regarding group extensions, we prove the following results in Section \ref{section:extensions} (see Theorems \ref{th:uitbreiding} and \ref{th:extension}).
\begin{theorem}
 Assume that $\Gamma$ is a group, equipped with some length function $l=l_{\Gamma}$, that fits in a short exact sequence
\[ 1 \rightarrow H \rightarrow \Gamma \stackrel{\pi}{\rightarrow} G \rightarrow 1. \]
Define a length function $l_G$ on $G$ by setting $l_G(\pi(x))=\inf\{l(y)\mid \pi(y)=\pi(x)\}$. If $G$ with the induced metric from $\Gamma$ has polynomial growth and if $H$ with the induced metric from $\Gamma$ has compression $\delta$, then the compression of $\Gamma$ is at least $\delta/4$.
\label{th:uitbreidingintro}
\end{theorem}

\begin{theorem}
Assume that $\Gamma$ is a finitely generated group, equipped with the word length function $l=l_{\Gamma}$ relative to some finite symmetric generating subset $S$ and that it  fits in a short exact sequence
\[ 1 \rightarrow H \rightarrow \Gamma \stackrel{\pi}{\rightarrow} G \rightarrow 1. \]
Equip $G$ with the word length function $l_G$ relative to $\pi(S)$. If $G$ is a finitely generated hyperbolic group in the sense of Gromov \cite{Gromovhyperbolicgroups} and if $H$, with the induced metric from $\Gamma$, has Hilbert space compression $\delta$, then the Hilbert space compression of $\Gamma$ is at least $\delta/5$.
\label{th:extensionintro}
\end{theorem}

We assume that every metric space in this article is a group and we assume that the metric is induced by a length function.

\section{Hilbert space compression for the limit of a directed system of groups \label{section:limit}}
Throughout this article, every metric space will be a group whose metric is induced by a length function. Let us start by recalling the definition of a length function on a group.
\begin{definition}
A length function $l$ on a group $G$ is a function $l:G\rightarrow \R^+$ satisfying
\begin{enumerate}
\item $l(x)=0 \Leftrightarrow x=1$,
\item $\forall x\in G, \ l(x)=l(x^{-1})$,
\item $\forall x,y\in G, \ l(xy)\leq l(x)+l(y)$.
\end{enumerate}
We say that $l$ is proper, whenever
  \[ \forall M \in \R^+: \ \mid \{ g\in G \mid l(g)\leq M \} \mid < \infty . \]
Every length function on $G$ induces a left-invariant metric on $G$ by $d(x,y)=l(x^{-1}y) \ \forall x,y\in G$.
\label{def:lengthfunction}
\end{definition}
Let $G_1\rightarrow G_2 \rightarrow G_3 \rightarrow \ldots$ be a directed system of groups such that the maps $G_i\rightarrow G_{i+1}$ are isometric injections. Denote $G$ the direct limit of this system. By 
definition, $G$ can be seen as the disjoint union of all the $G_i$ divided by some equivalence relation. Define the induced length function $l$ on $G$ by $l(x):= \lim_i l_{G_i}(x)$. We proceed under the assumption that $l$ is a proper length function on $G$. In this section, we ask ourselves the question how the Hilbert space compression of $G$, denoted by $\alpha(G)$, is related to the Hilbert space compressions of the $G_i$.\\

To begin, notice that every $G_i$ can be seen as a metric subspace of $G$ and so $\alpha(G)\leq \inf_{i\in \N}\alpha(G_i)$. Clearly, this bound is sharp, since as a family of subgroups we can take $G_i=G \ (\forall i\in \N)$. It proves more challenging to find a good lower bound for $\alpha(G)$. First, note that the same bound as above, i.e. $\inf_{i\in I}\alpha(G_i)$, is not always a lower bound. As an example, equip the group
\[ \Z^{(\Z)}=\{ f:\Z \rightarrow \Z \mbox{ with finite support} \}=\{(f,a) \in \Z\wr \Z \mid a=0 \} \]
with the induced metric from $\Z \wr \Z$. This group is the direct limit of the family of subgroups $G_n:=\Z^{2n+1}=\{ f:[-n,n]\rightarrow \Z\}$, where each $\Z^{2n+1}$ is equipped with the subspace metric from $\Z^{(\Z)}$. Since this metric is quasi-isometric to the standard word length metric on $\Z^{2n+1}$, we obtain $\Z^{(\Z)}$ as a limit of groups with compression $1$. However, it follows from the proof of Theorem $3.9$ in \cite{Arzhantseva} that $\Z^{(\Z)}$ has compression less than $\frac{3}{4}$.

Notice moreover that $\Z$ and $\Z^{(\Z)}$ have different compressions although they are both limits of groups of compression $1$. It will thus be necessary to include more information on how the groups $G_i$ are embedded in their respective Hilbert spaces in order to say something useful about the Hilbert space compression of their limit.

We propose the following
{\theorem \label{theorem:limit}
Assume that $G$ is the direct limit of a directed metric system $(G_i)_{i\in \N}$ of groups and that the induced length function $l$ is proper.
If $\inf_{i \in \N} (\alpha(G_i))=0$, then $\alpha(G)=0$.\\
Else, choose $0<\delta < \inf_{i\in \N}\alpha(G_i)$ and choose for every $i\in \N$, a Hilbert space $\h_i$, constants $C_i>0, \widetilde{C_i}, D_i,\widetilde{D_i}\geq 0$ and a map
\[ f_i:G_i \rightarrow \h_i \]
satisfying
\[ (1/C_i)\ d(x,y)^{\delta} -D_i \leq d(f_i(x),f_i(y)) \leq \widetilde{C_i}\ d(x,y) +\widetilde{D_i}  \ \forall x,y\in G_i. \]
Denote $g:\N \rightarrow \N$ such that for all $x\in G$ we have $x\in G_{g(n)}$ whenever $l(x)\leq \sqrt{n}$.
Then,
\[ \alpha(G)\geq \limsup_{n\rightarrow \infty}\frac{(\delta/2) \ln(n-1)}{\ln(C_{g(n)}\sqrt{2\ln(2)n} (\widetilde{C_{g(n)}} \ \sqrt{n} + \widetilde{D_{g(n)}}) +C_{g(n)} D_{g(n)})} .\]}


\begin{ex}
Assume that $G$ is an infinite direct sum of finite groups $G=F_0\oplus F_1\oplus F_2 \oplus \ldots$ where $F_0=\{1\}$. We can equip $G$ with a proper length function by setting $l(g)=\min\{n\in \N \mid g\in \oplus_{i=0}^n F_i\}$. Clearly, then $G$ is a directed system of metric spaces as above. Recalling the fact that finite groups have Hilbert space compression equal to $1$, we can apply Theorem \ref{theorem:limit}, obtaining $\alpha(G)=1$.
\end{ex}
\begin{ex}
Let $G\wr H$ be finitely generated and equip it with the word length metric relative to a finite symmetric generating subset.
Theorem \ref{theorem:limit} can be used as an easy way to estimate the compressions of spaces $G^{(H)}:=\{ f: H \rightarrow G \mid f$ has finite support $\}$, equipped with the induced length function from $G\wr H$. If $G$ is a discrete group with compression $\alpha$ and $H$ has polynomial growth of order $d$, then we obtain the lower bound $\frac{2\alpha}{d+4}$. It must be mentioned that the so obtained lower bound is weaker than the lower bound obtained in \cite{naoper}.
\end{ex}

Proposition \ref{prop:unifembedintro} from the Introduction plays a very important role in our proofs. It is implied by the following Proposition, which is Proposition $2.1$ of \cite{guedad}. We will give a (slightly modified version of) Guentner and Dadarlat's proof here because the details will be of vital importance further on.

{\proposition Let $X$ be a metric space. Then $X$ is uniformly embeddable in a Hilbert space if and only if for every $R>0$ and $\epsilon>0$ there exists a Hilbert space valued map $\xi:X\rightarrow \h, x\to \xi_x$ such that $\parallel \xi_x \parallel=1$ for all $x\in X$ and such that
\begin{enumerate}
\item $\sup \{\parallel \xi_x - \xi_{x'} \parallel: d(x,x') \leq R, x,x'\in X \} \leq \epsilon$,
\item $ \lim_{S\to \infty} \inf \{ \parallel \xi_x - \xi_{x'} \parallel : d(x,x')\geq S, x,x'\in X \} = \sqrt{2} $.
\end{enumerate} \label{prop:unifembed}}
\begin{proof} 
Assume that $X$ is uniformly embeddable and let $F:X\rightarrow \h$ be a uniform embedding of $X$ in a real Hilbert space $\h$. Let $\rho_-$ and $\rho_+$ be functions such that
\[ \rho_-(d(x,y)) \leq \parallel F(x)- F(y) \parallel \leq \rho_+(d(x,y)) .\]
Denote
\[ \Exp(\h)= \R \oplus \h \oplus (\h \otimes \h) \oplus (\h \otimes \h \otimes \h) \oplus \cdots \]
and define $\Exp:\h \rightarrow \Exp(\h)$ by
\[ \Exp(\zeta)=1 \oplus \zeta \oplus ( \frac{1}{\sqrt{2!}} \zeta \otimes \zeta ) \oplus (\frac{1}{\sqrt{3!}} \zeta \otimes \zeta \otimes \zeta ) \oplus \cdots .\]
Note that $\langle \Exp(\zeta),\Exp(\zeta') \rangle = e^{\langle \zeta,\zeta'\rangle}$, for all $\zeta,\zeta'\in \h$. For $t>0$ define 
\[ \xi_x=e^{-t \parallel F(x) \parallel^2} \Exp(\sqrt{2t} F(x)). \]
It is easily verified that $\langle \xi_x,\xi_x' \rangle = e^{-t \parallel F(x)- F(x') \parallel^2}$. Consequently, for all $x,x'\in X$ we have $\parallel \xi_x \parallel =1$, and
\begin{equation}
e^{-t \rho_+(d(x,x'))^2} \leq \langle \xi_x, \xi_{x'} \rangle \leq e^{-t \rho_-(d(x,x'))^2} .
\end{equation}
Putting $t=\frac{-\ln(1-\epsilon^2/2)}{\rho_+(R)^2}$,
it is easy to verify conditions $1$ and $2$ above.\\

Conversely, choose $p>0$ and assume that $X$ satisfies the conditions in the statement. There exist a sequence of maps $\eta_n:X\rightarrow \h_n$ and a sequence of numbers $S_0=0<S_1 <S_2 < \ldots$, increasing to infinity, such that for every $n\geq 1$ and every $x,x'\in X$,
\begin{enumerate}
\item $\parallel \eta_n(x) \parallel =1$
\item $\parallel \eta_n(x)- \eta_n(x') \parallel \leq \frac{1}{n^{1/2+p}}$, provided $d(x,x')\leq \sqrt{n}$,
\item $\parallel \eta_n(x)-\eta_n(x') \parallel \geq 1$, provided $d(x,x')\geq S_n$.
\end{enumerate}
Choose a base point $x_0 \in X$ and define $F:X\rightarrow \bigoplus_{n=1}^{\infty} \h_n$ by
\[ F(x)=\frac{1}{2} ((\eta_1(x)-\eta_1(x_0)) \oplus (\eta_2(x)-\eta_2(x_0)) \oplus \cdots ). \]
It is not hard to verify that $F$ is well defined and
\[ \rho_-(d(x,x'))\leq \parallel F(x)-F(x') \parallel \leq d(x,x')+C , \mbox{ for all } x,x'\in X ,\]
where $C>0$ is some constant, $\rho_-=\frac{1}{2} \sum_{n=1}^{\infty} \sqrt{n-1} \chi_{[S_{n-1},S_n)}$, and the $\chi_{[S_{n-1},S_n)}$ are the characteristic functions of the sets $[S_{n-1},S_n)$.

Indeed, let $x,x'\in X$. If $n$ is such that $\sqrt{n-1}\leq d(x,x')<\sqrt{n}$, we have
\begin{eqnarray*}
\parallel F(x)- F(x') \parallel ^2 &=& \frac{1}{4} \sum_{i\leq n-1} \parallel \eta_i(x) -\eta_i(x') \parallel^2 + \frac{1}{4} \sum_{i\geq n} \parallel \eta_i(x) -\eta_i(x') \parallel^2 \\
& \leq & (n-1) + \frac{1}{4} \sum_{i\geq n} \frac{1}{i^{1+2p}} \leq d(x,x')^2 + C
\end{eqnarray*}
where $C=\frac{1}{4} \sum_{i\geq n} \frac{1}{i^{1+2p}} < \infty$.

Similarly, if $n$ is such that $S_{n-1}\leq d(x,x') <S_n$, we have
\[ \parallel F(x)- F(x') \parallel ^2 \geq \frac{1}{4} \sum_{i\leq n-1} \parallel \eta_i(x)-\eta_i(x') \parallel^2 \geq \frac{n-1}{4}=\rho_-(d(x,x'))^2.\]

\end{proof}
The following corollary gives a connection between the compression of $X$ and the growth of the $(S_n)_{n\in \N_0}$.

\begin{corollary}
Assume that $X$ is a metric space with compression $\delta>0$. Choose $0<p<\delta$ and for all $n\in \N_0$, let
$R_n=n^r, \epsilon_n=\frac{1}{a n^b}$ for some $a,b,r \in \R_+$. For every $n$ large enough, we can find a collection of unit vectors $(\xi_n^x)_{x\in X}$ in some Hilbert space $\h_p$ satisfying
\begin{enumerate}
\item
$\parallel \xi_n^x- \xi_n^y \parallel \leq \epsilon_n$ provided $d(x,y)\leq R_n$,
\item
$\parallel \xi_n^x- \xi_n^y \parallel \geq 1$ provided $d(x,y)\geq n^{\frac{r+b+p}{\delta-p}}.$
\end{enumerate}
\label{cor:explicit}
\end{corollary}
\begin{proof}
Let $F:X\rightarrow \h$ be a uniform embedding of $X$ into a Hilbert space satisfying
\[ \forall x,y\in X: \ \frac{1}{C} d(x,y)^{\delta-p} -D \leq d(F(x),F(y)) \leq \widetilde{C} d(x,y) + \widetilde{D},\]
for some $C,\widetilde{C} >0$ and $D,\widetilde{D} \geq 0$.
Denote $\rho_-(d(x,y)):= \frac{1}{C} d(x,y)^{\delta-p} -D $ and $ \rho_+(d(x,y)):=\widetilde{C} d(x,y) + \widetilde{D}$. 
In the proof of Proposition \ref{prop:unifembed}, set $t_n=\frac{-\ln(1-\frac{\epsilon_n^2}{2})}{\rho_+(R_n)^2}$ and obtain vectors $(\xi_n^x)_{x\in X}$ such that
\[ e^{-t_n \rho_+(d(x,x'))^2} \leq \langle \xi_n^x, \xi_n^{x'} \rangle \leq e^{-t_n \rho_-(d(x,x'))^2} .\]
It is easy to verify that the vectors $(\xi_n^x)_{x\in X}$ satisfy Condition $(1)$ of this Corollary. Regarding the second condition, note that
\begin{eqnarray*}
\parallel \xi_n^x - \xi_n^{y} \parallel ^2 &=& 2-2\langle \xi_n^x, \xi_n^{y} \rangle \\
&\geq & 2-2e^{-t_n \rho_-(d(x,y))^2} \\
&=& 2-2e^{\frac{\ln(1-\frac{\epsilon_n^2}{2})}{(\widetilde{C}\ R_n +\widetilde{D})^2}((1/C)\ d(x,y)^{\delta-p} -D)^2}.
\end{eqnarray*}
Consequently we have $\parallel \xi_n^x - \xi_n^y \parallel\geq 1$ whenever
\[ 2-2e^{\frac{\ln(1-\frac{\epsilon_n^2}{2})}{(\widetilde{C}\ R_n +\widetilde{D})^2}((1/C)\ d(x,y)^{\delta-p} -D)^2} \geq 1, \]
i.e. whenever
\[ (1-\frac{\epsilon_n^2}{2})^{\frac{((1/C) d(x,y)^{\delta-p} -D)^2}{(\widetilde{C} R_n + \widetilde{D})^2}} \leq \frac{1}{2} . \]
This is true if and only if
\[ \frac{(1/C) d(x,y)^{\delta-p} - D}{\widetilde{C} R_n + \widetilde{D}}\geq \sqrt{\frac{-\ln(2)}{\ln(1-\frac{\epsilon_n^2}{2})}}, \]
if and only if
\[ d(x,y) \geq [C (\sqrt{\frac{-\ln(2)}{\ln(1-\frac{1}{2a^2n^{2b}})}} (\widetilde{C} n^r + \widetilde{D})+D)]^{\frac{1}{\delta-p}} .\]

Since $\frac{-\ln(2)}{\ln(1-\frac{1}{2a^2n^{2b}})} \leq \ln(2) 2 a^2n^{2b} \leq 2a^2n^{2b}$, it suffices to take
\[ d(x,y)\geq [C \sqrt{2} an^b (\widetilde{C} n^r + \widetilde{D})+CD]^{\frac{1}{\delta-p}}:=A_n .\]
If $n$ is large enough, then $A_n\leq n^\frac{b+r+p}{\delta-p}$, so we obtain 
$\parallel \xi_n^x- \xi_n^y \parallel \geq 1$ provided $d(x,y)\geq n^{\frac{r+b+p}{\delta-p}}$.
\end{proof} 
\begin{remark}
In the second part of the proof of Proposition \ref{prop:unifembed}, we need the condition
\begin{equation}
\parallel \eta_n(x)- \eta_n(x') \parallel \leq \frac{1}{n^{1/2+p}}\mbox{, provided } d(x,x')\leq \sqrt{n},
\label{eq:conditiequageo}
\end{equation}
to prove that $F$ is Lipschitz. Assume now that $X$ is a finitely generated group which is equipped with the word length metric relative to some finite symmetric generating subset. Then $X$ is a geodesic metric space implying that any function $\rho_+$ satisfying $\forall x,y \in X: \parallel F(x)-F(y) \parallel \leq \rho_+(d(x,y))$, can assumed to be of the form $Cd(x,y)+D$ for some constants $C,D\geq 0$. Therefore, in order to prove that $F$ is Lipschitz, we can relax condition (\ref{eq:conditiequageo}) to
\[  \parallel \eta_n(x)- \eta_n(x') \parallel \leq \frac{1}{n^{1/2+p}}\mbox{, provided } d(x,x')\leq \ln(n). \]
The condition 
\[ \parallel \eta_n(x)-\eta_n(x') \parallel \geq 1 \mbox{, provided }d(x,x')\geq S_n, \]
then holds for smaller $S_n$ and we obtain that the function $\rho_-=\frac{1}{2} \sum_{n=1}^{\infty} \sqrt{n-1} \chi_{[S_{n-1},S_n)}$ becomes larger. This will help us to get better compression estimates later.
\label{remark:quageo}
\end{remark}

{\bf Proof of Theorem \ref{theorem:limit}}
Choose $n\in \N_0, p>0$ and denote $R=\sqrt{n},\epsilon=\frac{1}{n^{1/2+p}}$. Next, take $g(n)\in \N$ such that $x\in G_{g(n)}$ whenever $l_G(x)\leq R=\sqrt{n}$. Set $t=\frac{-\ln(1-\epsilon^2/2)}{(\widetilde{C_{g(n)}}\ R +\widetilde{D_{g(n)}})^2}$ and take vectors $(\xi_x)_{x\in G_{g(n)}}$ as in the proof of proposition \ref{prop:unifembed}, i.e. such that for all $x,y\in G_{g(n)}:$
\[ e^{-t (\widetilde{C_{g(n)}}\ d_{g(n)}(x,y) +\widetilde{D_{g(n)}})^2} \leq \langle \xi_x, \xi_{y} \rangle \leq e^{-t ((1/C_{g(n)})\ d_{g(n)}(x,y)^{\delta} -D_{g(n)})^2}. \]
From the lower bound on $\langle \xi_x, \xi_{y} \rangle$, one derives
\[ \parallel \xi_x - \xi _y \parallel \leq \epsilon \mbox{ whenever } d_{g(n)}(x,y)\leq R. \]
Calculating as in Corollary $2.6$, we derive that
$\parallel \xi_x - \xi_y \parallel \geq 1$ whenever $d_{g(n)}(x,y)\geq S_n:=[C_{g(n)} (\sqrt{\frac{-\ln(2)}{\ln(1-\frac{1}{2n^{2p+1}})}} (\widetilde{C_{g(n)}} \sqrt{n} + \widetilde{D_{g(n)}})+D_{g(n)})]^{\frac{1}{\delta}}$.

In the proof of Proposition $3.1$ of \cite{guedad}, Dadarlat and Guentner explain how the family $(\xi_x)_{x\in G_{g(n)}}$ can be extended to a family of unit vectors $(\hat{\xi}_x)_{x\in G}$ in a larger Hilbert space, but still satisfying similar inequalities. More precisely, we obtain unit vectors $(\hat{\xi}_x)_{x\in G}$ in a Hilbert space satisfying
\begin{enumerate}
\item $\parallel \xi_x - \xi _y \parallel \leq \frac{1}{n^{1/2+p}} \mbox{ whenever } d(x,y)\leq \sqrt{n}; $
\item $\parallel \xi_x - \xi _y \parallel \geq 1 \mbox{ whenever }d(x,y)\geq S_n. $
\end{enumerate}
From the proof of proposition \ref{prop:unifembed}, we derive the existence of a large-scale uniform embedding of $G$ into a Hilbert space whose compression function $\rho_-$, is greater than $\frac{1}{2} \sum_{n=1}^{\infty} \sqrt{n-1} \chi_{[S_{n-1},S_n)}(t)$. Choose some $\beta\in [0,1]$, and define $\gamma: \R^+ \rightarrow \R^+, t\mapsto t^{\beta}$. If $\gamma$ eventually lies under some multiple of $\rho_-$, then the compression of $G$ is greater than $\beta$. There exists $T, \overline{C} \in \R^+$ such that $\gamma(t)\leq \overline{C} \rho_-(t), \ \forall t\geq T$ if 
\[ \limsup_{n\rightarrow \infty}\frac{S_n^{\beta}}{\sqrt{n-1}} < \infty,\]
if and only if
\[\beta \leq \limsup_{n\rightarrow \infty}\{\frac{ (\delta/2) \ln(n-1)}{\ln(C_{g(n)} (\sqrt{\frac{-\ln(2)}{\ln(1-\frac{1}{2n^{2p+1}})}} (\widetilde{C_{g(n)}} \sqrt{n} + \widetilde{D_{g(n)}})+D_{g(n)}))} \} .\]
Recalling that $\lim_{n\to \infty} [(\frac{-\ln(2)}{\ln(1-\frac{1}{2n^{2p+1}})}) / (2\ln(2)n^{2p+1})] =1 $ and that we can let $p$ go to $0$ since it is just a positive real number that we've chosen, we get the desired lower bound
\[  \limsup_{n\rightarrow \infty} \frac{(\delta/2) \ln(n-1)}{\ln(C_{g(n)} \sqrt{2\ln(2) n}(\widetilde{C_{g(n)}} \sqrt{n} + \widetilde{D_{g(n)}})+C_{g(n)} D_{g(n)})} , \]
for the Hilbert space compression of $G$. $\Box$ \\

\begin{remark}
If $G$ happens to be a quasi-geodesic space, then we can use Remark \ref{remark:quageo} to improve our result. Using the same notations as in Theorem \ref{theorem:limit} and assuming that $G$ is a quasi-geodesic space, we obtain the following.
\begin{quote}  Denote $g:\N \rightarrow \N$ a function such that for all $x\in G$ we have $x\in G_{g(n)}$ whenever $l(x)\leq \ln(n)$.
Then,
 \[ \alpha(G)\geq \limsup_{n\rightarrow \infty}\frac{ (\delta/2) \ln(n-1)}{\ln(C_{g(n)}\sqrt{2\ln(2)n} (\widetilde{C_{g(n)}} \ \ln(n)+ \widetilde{D_{g(n)}}) +C_{g(n)} D_{g(n)})}.\]
\end{quote}
\end{remark}

\begin{remark}
All of the above easily generalizes to directed systems of groups $(G_i)_{i\in I}$ where $I$ is {\em any} directed set.
\end{remark}

\section{Hilbert space compression for group extensions \label{section:extensions}}
In this paragraph, $\Gamma$ will denote a group whose metric is induced by a length function $l_{\Gamma}$ and $H$ will denote a normal subgroup of $\Gamma$ which has strictly positive Hilbert space compression when equipped with the induced metric $l_H:=(l_{\Gamma})_{\mid H}$. We assume that $\Gamma$ is a group extension
\[ 1 \rightarrow H \rightarrow \Gamma \stackrel{\pi}{\rightarrow} G \rightarrow 1 ,\]
and we equip $G$ with the induced length from $\Gamma$, i.e. for all $x\in \Gamma: l_G(\pi(x)):=\inf\{ l_\Gamma(y) \mid y\in \Gamma, \pi(y)=\pi(x) \}.$ Given certain conditions on $G$, we shall give bounds on the Hilbert space compression of $\Gamma$.\subsection{Extensions by a group of polynomial growth}
Let us begin by recalling the definition of a metric space with polynomial growth.
\begin{definition}
A metric space $X$ has polynomial growth if there exists a polynomial $P$ such that $\mid \overline{B(x,R)} \mid \leq P(R)$ for every $x\in X$ and every $R\geq 0$. Here $\overline{B(x,R)}$ is the closed ball with radius $R$ and center $x$.
\end{definition}
Notice that $G$ can have polynomial growth only if $l_G$ is proper.
In Lemma $6.6$ of \cite{Tu}, Tu proves that groups of polynomial growth have property A. We obtain the following lemma by quantifying his proof.
\begin{lm}
Let $G$ be a group, equipped with a proper length function, that has polynomial growth. Let $p\in ]0,1[$ be any real number. There exists $n_0\in \N$ such that for every natural number $n\geq n_0$, there exists a collection of unit vectors $(g_n(x))_{x\in G}$ in $l^2(G)$ such that 
$\parallel g_n(x) \parallel_2=1  ,\ \forall x \in G$ and
\begin{enumerate}
\item $\mid 1- \langle g_n(x),g_n(y) \rangle \mid \leq \frac{1}{4n^{1+2p}}$ provided $d_G(x,y)\leq \sqrt{n}$,
\item {\em supp}$(g_n(x))\subset B(x,S_n^G)$ for all $x\in G$ where $S_n^G=n^{3/2+5p}.$
\end{enumerate}
\label{lm:directAgrowth}
\end{lm}
\begin{proof}
 For each $x\in G$ and $r\in \R$, denote by $B(x,r)\subset G$ the ball of radius $r$ and center $x$. Denote the characteristic function of $B(x,r)$ by $\chi_x^{r}$. We shall denote $B(1,r)$ simply by $B_r$ and $\chi_1^r$ by $\chi_r$. For $n\in \N_0$, denote $R_n=\sqrt{n}$ and, with the convention that $\forall a\in \R: a/0=\infty$, let $k(n)$ be the infimum of all real numbers $r$
such that 
\[ \frac{ \mid B_{r+R_n} \mid }{ \mid B_{r-R_n} \mid } \leq 1 + \frac{1}{2n^{1+2p}} .\]
Clearly, such $k(n)$ exists, since if it didn't exist, then $\forall i\in \N_0$,
\[ \mid B_{2iR_n+R_n} \mid \geq \mid B_{(2i-1)R_n} \mid (1 + \frac{1}{2n^{1+2p}})  \geq \ldots \geq \mid B_{R_n} \mid (1 + \frac{1}{2n^{1+2p}})^{i}, \]
obtaining a contradiction since the left hand side depends polynomially on $i$ whereas the right hand side depends exponentially on $i$.
 
We claim that there exists $\overline{n}\in \N_0$ such that $\forall n\geq \overline{n}: k(n)\leq 2n^{3/2+4p}$. 
 Assume therefore, that such $\overline{n}$ does not exist. Then there exists a strictly monotone increasing sequence $(n_i)_{i\in \N}$ such that
\[ \forall i: \frac{\mid B_{2n_i^{3/2+4p}+R_{n_i}} \mid }{\mid B_{2n_i^{3/2+4p}-R_{n_i}} \mid } \geq 1 + \frac{1}{2n_i^{1+2p}} .\]
Denoting the integer part of a real number $a$ by $[a]$ and assuming for the last inequality below that $\forall i: n_i \geq 2^{1/p}$, we obtain that
\begin{eqnarray*}
\mid B_{2n_i^{3/2+4p}+R_{n_i}} \mid & \geq & (1 + \frac{1}{2n_i^{1+2p}} )\mid B_{2n_i^{3/2+4p}-2R_{n_i} + R_{n_i}} \mid \\
 & \geq  &  (1 + \frac{1}{2n_i^{1+2p}} )^2 \mid B_{2n_i^{3/2+4p}-4R_{n_i} + R_{n_i}} \mid  \\
& \geq & \ldots  \\
& \geq & (1 + \frac{1}{2n_i^{1+2p}} )^{[n_i^{1+4p}]} \mid B_{R_{n_i}}  \mid \\
& \geq & (1 + \frac{1}{2n_i^{1+2p}} )^{n_i^{1+3p}} \mid B_{R_{n_i}}  \mid .
\end{eqnarray*}
Since $\lim_{i \to \infty} (1+ \frac{1}{2n_i^{1+2p}})^{n_i^{1+2p}}= \exp(1/2)$, it is clear that the right hand side depends exponentially on $n_i$, whereas the left hand side depends polynomially on $n_i$. We obtain a contradiction.

Consider now the functions $\chi_x^{k(n)}$. They are elements of $l^1(G)$ such that $d_G(x,y)\leq \sqrt{n}=R_n$ implies 
 \[ \frac{\parallel \chi_x^{k(n)} - \chi_{y}^{k(n)} \parallel_1}{\parallel \chi_x^{k(n)} \parallel_1} \leq \frac{\mid B(x,k(n)+R_n) \mid - \mid B(x,k(n)-R_n) \mid }{\mid B(x,k(n)-R_n) \mid }= \frac{\mid B_{k(n)+R_n} \mid }{\mid B_{k(n)-R_n} \mid } -1 \leq \frac{1}{2n^{1+2p}} .\]
Moreover, the support of $\chi_x^{k(n)}$ lies inside $\overline{B(x,k(n))}\subset B(x,2k(n))\subset B(x,4n^{3/2+4p})\subset B(x,n^{3/2+5p})$ whenever $n$ is larger than some natural number $n_0$.
To conclude, take $n\geq \max(n_0,\overline{n})$ and define $g_n(x)= \sqrt{\frac{\chi_x^{k(n)}}{\parallel \chi_{k(n)} \parallel_1}}$. Clearly, these are elements of norm $1$ in $l^2(G)$ that satisfy condition $(2)$ of this lemma. To show that they also satisfy condition $(1)$, take $x,y$ such that $d_G(x,y)\leq R_n$. Then
\begin{eqnarray*}
\parallel g_n(x) - g_n(y) \parallel_2 ^2 &= & \sum_{z\in X} \mid g_n(x)(z) - g_n(y)(z) \mid ^2 \\
& \leq & \sum_{z\in X} (\mid g_n(x)(z) - g_n(y)(z) \mid \cdot \mid g_n(x)(z) + g_n(y)(z) \mid ) \\
& = & \sum_{z\in X} \mid g_n(x)(z)^2 - g_n(y)(z)^2 \mid \\
& = &\frac{\parallel \chi_x^{k(n)} - \chi_y^{k(n)} \parallel_1}{\parallel \chi_{k(n)} \parallel_1 } \leq \frac{1}{2n^{1+2p}}.
\end{eqnarray*}
Therefore $\mid 1- \langle g_n(x) , g_n(y) \rangle \mid \leq \frac{1}{4n^{1+2p}}$ as desired.
\end{proof}
\noindent
\begin{theorem}
 Assume that $\Gamma$ is a group, equipped with some length function $l=l_{\Gamma}$, that fits in a short exact sequence
\[ 1 \rightarrow H \rightarrow \Gamma \stackrel{\pi}{\rightarrow} G \rightarrow 1. \]
If $G$ with the induced metric from $\Gamma$ has polynomial growth and if $H$ with the induced metric from $\Gamma$ has strictly positive Hilbert space compression, then the Hilbert space compression of $\Gamma$ is strictly positive. More precisely, if the compression of $H$ equals $\delta$, then the compression of $\Gamma$ is at least $\delta/4$.
\label{th:uitbreiding}
\end{theorem}
\begin{proof}
Denote the Hilbert space compression of $H$ by $\delta>0$ and choose $0<p<\delta$.
If $n$ is large enough, then $n^{3/2+6p}\geq 2S_n^G +\sqrt{n}=2n^{3/2+5p} +\sqrt{n}$. For $n$ {\em sufficiently large}, Corollary $2.6$ applied for $r={3/2+6p}, a=\sqrt{2}, b=1/2+p$ gives a Hilbert space $\h$ and maps $h_n:H\rightarrow \h$ such that $\parallel h_n(s) \parallel = 1 \ \forall s\in H$ and 
\begin{itemize}
\item $\mid 1-\langle h_n(s),h_n(\tilde{s}) \rangle \mid \leq \frac{1}{4n^{1+2p}} \mbox{ provided } d_H(s,\tilde{s})\leq 2S_n^G+\sqrt{n}$
\item $\mid \langle h_n(s),h_n(\tilde{s}) \rangle \mid \leq \frac{1}{2}$ whenever $d(s,\tilde{s})\geq S_n^H:= n^\frac{2+8p}{\delta-p}=(n^{1/2+3p} S_n^G)^\frac{1}{\delta-p}$.
\end{itemize}

For $n$ sufficiently large, Lemma \ref{lm:directAgrowth} provides maps $g_n:G\rightarrow l^2(G)$ such that $\parallel g_n(x) \parallel_2=1  ,\ \forall x \in G$ and such that 
\begin{itemize}
\item $\mid 1-\langle g_n(x),g_n(y) \rangle \mid \leq \frac{1}{4n^{1+2p}}$ provided $d_G(x,y)\leq \sqrt{n}$,
\item supp$(g_n(x))\subset B(x,S_n^G)$ for all $x\in G$ where $S_n^G=n^{3/2+5p}.$
\end{itemize}
In the proof of Theorem $4.1$ in \cite{guedad}, Guentner and Kaminker fix $n$ and use the maps $g_n$ and $h_n$ to construct a map $f_n:\Gamma \rightarrow l^2(G,\h)$ such that $\parallel f_n(a) \parallel = 1  ,\ \forall a\in \Gamma$ and
\begin{itemize}
\item $\mid 1-\langle f_n(a),f_n(b) \rangle \mid \leq \frac{1}{2n^{1+2p}}$ if $d(a,b)\leq \sqrt{n}$,
\item $\parallel f_n(a)-f_n(b) \parallel  \geq 1 $ if $d(a,b)\geq 2S_n^G+S_n^H$.
\end{itemize}
Denoting $\overline{S_n}=n^pS_n^H$, we obtain for $n$ larger than some $n_1\in \N_0$ that
\begin{itemize}
\item $\parallel f_n(a)-f_n(b) \parallel \leq \frac{1}{n^{1/2+p}}$ if $d(a,b)\leq \sqrt{n}$,
\item $ \parallel f_n(a)-f_n(b) \parallel  \geq 1$ if $d(a,b)\geq \overline{S_n}$.
\end{itemize}
From the second part of the proof of Proposition \ref{prop:unifembed}, we find a Lipschitz uniform embedding $F$ of $\Gamma$ into a Hilbert space such that 
$\parallel F(x)-F(y) \parallel \geq \overline{\rho}_-(d(x,y)):= \frac{1}{2} \sum_{n=n_1+1}^\infty \sqrt{n-n_1} \chi_{[\overline{S_{n-1}},\overline{S_n}[}(d(x,y))$ for $d(x,y)\in [\overline{S_{n_1}}, \infty [$. A function $u:\R^+ \rightarrow \R^+$ of the form $x\mapsto x^\beta$ for some $\beta>0$ eventually lies under a constant $\overline{C}$ times $\overline{\rho}_-$ whenever there exists $N$ such that for all $n\geq N, \overline{S_{n}}^{\beta}\leq \overline{C} \sqrt{n-n_1}$. For this to be true, we need that $\overline{S_{n}}^{\beta}$, as a function in $n$ is of degree less than $\frac{1}{2}$. Remember that
\[ \overline{S_{n}}^{\beta}= n^{\beta p} n^{\frac{(2+8p)\beta}{\delta-p}}, \]
which is of degree $(\frac{2+8p}{\delta-p}+p)\beta$. 
 The compression of $\Gamma$ is now greater than $\beta$ whenever
 \[ \beta \leq  [2(\frac{2+8p}{\delta-p}+p)]^{-1}.\]
Since $p$ was just any number between $0$ and $\delta$, we decide to let $p$ go to $0$, obtaining that the compression of $\Gamma$ is at least $\delta/4$.
 \end{proof}
 
\subsection{Extensions by a finitely generated word-hyperbolic group}
We start by a lemma similar to Lemma \ref{lm:directAgrowth}.
More precisely, quantifying the proof of Proposition $8.1$ in \cite{Tu}, we obtain
\begin{lm}
Let $G$ be a finitely generated hyperbolic group. Let $p\in ]0,1[$ be any real number. There exists $n_0\in \N$ such that for every natural number $n\geq n_0$, there exists a collection of unit vectors $(g_n(x))_{x\in G}$ in $l^2(G)$ such that 
$\parallel g_n(x) \parallel=1  ,\ \forall x \in G$ and
\begin{itemize}
\item $\mid 1- \langle g_n(x),g_n(y) \rangle \mid \leq \frac{1}{4n^{1+2p}}$ provided $d_G(x,y)\leq \ln(n)$,
\item supp$(g_n(x))\subset B(x,S_n^G)$ for all $x\in G$ where $S_n^G=n^{2+6p}.$
\end{itemize}
\label{lm:directAhyp}
\end{lm}
\begin{proof}
Choose $q>0$ such that $(2+5p)(1/2-q)> 1+2p$ and choose $a\in \partial G$, the Gromov boundary of $G$. For all $x\in X$, let $[[x,a[[$ be the set of infinite geodesics from $x$ to $a$, i.e. isometries $g:\N \rightarrow X$ such that $g(0)=x$ and $\lim_{n\to \infty} g(n)=a$. For every $x\in G$ and $k,n\in \N_0$, we define elements of $l^1(G)$ as follows:
\[ F(x,k,n) = \mbox{ characteristic function of } \bigcup_{d(x,y)<k, g\in [[y,a[[} g([n,2n]) ,\]
\[ H(x,n) = \frac{1}{n^{3/2-q}} \sum_{k<\sqrt{n}} F(x,k,n) .\]
Define $k(n)=n^{2+5p}$.
Following the proof of Proposition $8.1$ in \cite{Tu}, we find constants $C,D >0$ such that the $H(x,n)_{x\in G, n\in \N_0}$ satisfy the following conditions for all $n$ greater than some natural number $n_0$:
\begin{enumerate}
\item $\parallel H(x,k(n)) \parallel_1 \geq 1$
\item $\parallel H(x,k(n)) - H(y,k(n)) \parallel_1 \leq \frac{C\ln(n)+D}{(k(n))^{1/2-q}} \leq \frac{1}{4n^{1+2p}}$ provided $d(x,y)\leq \ln(n)$
\item supp$(H(x,k(n))) \subset B(x,n^{2+6p})$.
\end{enumerate}
For all $x\in G, n\geq n_0$, set $g(x,n)=\sqrt{\frac{H(x,k(n))}{\parallel H(x,k(n)) \parallel_1}}$ to obtain a collection of elements of $l^2(G)$. Calculating as in the end of the proof of Lemma \ref{lm:directAgrowth} and using the fact that $\forall z\in G: H(x,n)(z)\geq 0$, we obtain
\begin{eqnarray*}
\parallel g(x,n) - g(y,n) \parallel_2 ^2 & \leq & \parallel \frac{H(x,k(n))}{\parallel H(x,k(n)) \parallel_1} - \frac{H(y,k(n))}{\parallel H(y,k(n)) \parallel_1} \parallel_1 \\
& \leq & \frac{ \parallel H(x,k(n)) - H(y,k(n)) \parallel_1}{\parallel H(x,k(n)) \parallel_1} \\
& & + \parallel H(y,k(n)) \parallel_1 (\frac{1}{\parallel H(x,k(n)) \parallel_1} - \frac{1}{\parallel H(y,k(n)) \parallel_1}) \\
& \leq & 2\frac{\parallel H(x,k(n))-H(y,k(n)) \parallel_1 }{\parallel H(x,k(n)) \parallel_1 } \\
& \leq & 2 \parallel H(x,k(n))-H(y,k(n)) \parallel_1  .
\end{eqnarray*}
The $g(x,n)$ with $n\geq n_0$ satisfy the conditions of this Lemma.
\end{proof}
\begin{theorem}
Assume that $\Gamma$ is a finitely generated group, equipped with the word length function $l=l_{\Gamma}$ relative to some finite symmetric generating subset $S$ and that it  fits in a short exact sequence
\[ 1 \rightarrow H \rightarrow \Gamma \stackrel{\pi}{\rightarrow} G \rightarrow 1. \]
Equip $G$ with the word length function $l_G$ relative to $\pi(S)$. If $G$ is a finitely generated hyperbolic group in the sense of Gromov \cite{Gromovhyperbolicgroups} and if $H$, with the induced metric from $\Gamma$, has Hilbert space compression $\delta$, then the Hilbert space compression of $\Gamma$ is at least $\delta/5$.
\label{th:extension}
\end{theorem}
\begin{proof}
The proof is analogous to the proof of Theorem \ref{th:uitbreiding}. First, denoting $S_n^G=n^{2+6p}$ as in Lemma \ref{lm:directAhyp} and using Corollary \ref{cor:explicit} for $r=2+7p, a=\sqrt{2}, b=(1/2)+p$, we find a Hilbert space $\h$ and unit vectors $(h_n(s))_{s\in H}\in \h$ for every $n$ large enough such that
\begin{itemize}
\item $\mid 1-\langle h_n(s),h_n(\tilde{s}) \rangle \mid \leq \frac{1}{4n^{1+2p}} \mbox{ provided } d_H(s,\tilde{s})\leq 2S_n^G+\sqrt{n}$
\item $\mid \langle h_n(s),h_n(\tilde{s}) \rangle \mid \leq \frac{1}{2}$ whenever $d(s,\tilde{s})\geq S_n^H:=n^\frac{(5/2)+9p}{\delta-p}=(n^{1/2+3p}S_n^G)^{\frac{1}{\delta-p}}$.
\end{itemize}

Recall from Lemma \ref{lm:directAgrowth} that, for $n$ large enough, we've proven the existence of maps $g_n:G\rightarrow l^2(G)$ such that $\parallel g_n(x) \parallel_2=1  ,\ \forall x \in G$ and such that 
\begin{itemize}
\item $\mid 1-\langle g_n(x),g_n(y) \rangle \mid \leq \frac{1}{4n^{1+2p}}$ provided $d_G(x,y)\leq \ln(n)$,
\item supp$(g_n(x))\subset B(x,S_n^G)$ for all $x\in G$ where $S_n^G=n^{2+6p}.$
\end{itemize}
The proof of Theorem $4.1$ in \cite{guedad} gives for every $n\in \N$ larger than some $n_1\in \N$, a map $f_n:\Gamma \rightarrow l^2(G,\h)$ such that $\parallel f_n(a) \parallel = 1  ,\ \forall a\in \Gamma$ and
\begin{itemize}
\item $\parallel f_n(a)-f_n(b) \parallel \leq \frac{1}{n^{1/2+p}}$ if $d(a,b)\leq \ln(n)$,
\item $ \parallel f_n(a)-f_n(b) \parallel  \geq 1$ if $d(a,b)\geq \overline{S_n}:=n^pS_n^H$.
\end{itemize}
Remark \ref{remark:quageo} implies the existence of a Lipschitz uniform embedding $F$ of $\Gamma$ into a Hilbert space such that 
$\parallel F(x)-F(y) \parallel \geq \overline{\rho}_-(d(x,y)):= \frac{1}{2} \sum_{n=n_1+1}^\infty \sqrt{n-n_1} \chi_{[\overline{S_{n-1}},\overline{S_n}[}(d(x,y))$ for $d(x,y)\in [\overline{S_{n_1}}, \infty [$. A function $u:\R^+ \rightarrow \R^+$ of the form $x\mapsto x^\beta$ for some $\beta>0$ eventually lies under a constant $\overline{C}$ times $\overline{\rho}_-$ whenever there exists $N$ such that for all $n\geq N, \overline{S_{n}}^{\beta}\leq \overline{C} \sqrt{n-n_1}$. For this to be true, we need that $\overline{S_{n}}^{\beta}$, as a function in $n$ is of degree less than $\frac{1}{2}$. 
 Therefore, the compression of $\Gamma$ is greater than $\beta$ whenever
 \[ \beta \leq  [2(\frac{(5/2)+9p}{\delta-p}+p)]^{-1}.\]
 Since this is true for any arbitrarily small $p$, we decide to let $p$ go to $0$, obtaining that the compression of $\Gamma$ is at least $\delta/5$.
 \end{proof}
 
Alain Valette pointed out to the author that a stronger result is valid in the following special case.
 \begin{theorem}
 Let $A$ and $G$ be finitely generated groups, each equipped with the word length metric relative to a finite symmetric generating subset. Assume that $A$ is abelian, that $G$ is a finitely generated word-hyperbolic group and that
 \[ 0 \rightarrow H \rightarrow \Gamma \rightarrow G \rightarrow 1 ,\]
 is a central extension. The compression of $\Gamma$, equipped with the word length metric relative to a finite symmetric generating subset, equals $1$.
 \label{th:hyperbolicspecialcase}
 \end{theorem}
 \begin{proof}
 Denote the second bounded cohomology group of $G$, defined using bounded cocycles, by $H^2(G,H)$. By \cite{Neumann}, the comparison map
 \[ H_b^2(G,H) \rightarrow H^2(G,H) ,\]
 is onto for $G$ hyperbolic, i.e. every $2$-cocycle has a bounded representative.

Now, let $s:G\rightarrow \Gamma$ be a (set-theoretic) section, i.e. $p\circ s = Id_G$ and define
\[ c(x,y)= s(xy)^{-1} s(x)s(y) \ \forall x,y \in G .\]
By the above, we can assume that $c$ is bounded and so Gersten's result implies that $\Gamma$ is quasi-isometric to $G\times H$. Consequently, the compression of $\Gamma$ equals the minimum of the compressions of $G$ and $H$ \cite{guekam}, which is $1$ (see \cite{tess}, \cite{brodskiy}).
 
 \end{proof}
\noindent
{\bf Acknowledgements}\\

\noindent
I thank Alain Valette for reading previous versions of this article, for introducing me into the world of Hilbert space compression, for encouragements and for pointing out the result formulated in Theorem \ref{th:hyperbolicspecialcase}. I thank Paul Igodt for encouragements and his support. I thank Pieter Penninckx for useful suggestions and critical remarks and I thank Nansen Petrosyan and Pierre de la Harpe for very interesting conversations.

\end{document}